\documentclass[12pt]{article}

\usepackage[dvipsnames]{xcolor}
\usepackage{graphicx}   
\usepackage{amssymb,mathrsfs,dsfont}
  \let\theoremstyle\relax
 \let\theoremstyle\relax

\usepackage{amsthm}
\usepackage{units}
\usepackage{amsmath}

\usepackage{comment}

\newcommand{\be}{\begin{equation}}
\newcommand{\ee}{\end{equation}}
\newcommand{\beq}{\begin{equation}}
\newcommand{\eeq}{\end{equation}}

\newtheorem{lemma}{Lemma}

\newtheorem{theorem}{Theorem}
\newtheorem{remark}{Remark}

\newtheorem{example}{Example}

\colorlet{GREEN}{green}
\colorlet{RED}{red}
\colorlet{BLUE}{blue}
\colorlet{MAGENTA}{magenta}
\colorlet{BROWN}{brown}

\title{Controllable Sequences of Minimal Length for Discrete-Time Switched Linear Control Systems.}

\author{Paolo Mason and Antoine Girard\thanks{Universit\'e Paris-Saclay, CNRS, CentraleSup\'elec,  Laboratoire des signaux et syst\`emes, 91190, Gif-sur-Yvette, France, {\tt paolo.mason@centralesupelec.fr, antoine.girard@centralesupelec.fr}}}

\begin{document}


\maketitle

\begin{abstract}
In this paper, we provide a novel characterization of the reachable set of discrete-time switched linear control systems and a Kalman-type criterion for controllability, assuming that the switching parameter can be used as a control parameter in addition to the actual control variable. For controllable switched linear control systems it turns out that there always exists a switching sequence such that the reachable set of the corresponding linear time-variant system covers the whole state space after a sufficiently large time.
We provide estimates on the minimal time guaranteeing this property in terms of the state dimension, number of modes and rank of the control matrices, and show that such estimates are actually tight in some relevant cases.
\end{abstract}

%
%

\section{Introduction}

In this paper, we study the controllability of discrete-time switched linear control systems in $\mathbb{R}^n$ of the form
\begin{equation} 
\label{slcs}
x_{k} = A_{i_k}x_{k-1}+ B_{i_k} u_k,\qquad k\geq 1,
\end{equation} 
with $m$ modes corresponding to the pairs $(A_i,B_i),\ i=1,\dots,m$. Namely, we focus on the analysis of controllable switching sequences $i_1,\dots,i_\ell$, that is, switching sequences of finite length $\ell\ge 1$ such that every point of $\mathbb{R}^n$ can be reached at time $\ell$ starting from the origin by a solution of \eqref{slcs} corresponding to a suitable choice of $u_1,\dots,u_\ell$. Of particular interest in this work is the estimation of (tight) bounds on the length of the shortest controllable sequences.

Early results can be found in~\cite{conner1984state}, where the authors observe that the number of time steps needed to cover the entire reachable set may be larger than the dimension of the space and estimate such a number, in some particular cases, in terms of the number of modes and the dimension of the state and input variables. This work can be seen as an extension of~\cite{stanford1980controllability}, where the authors focus primarily on the structure of the reachable set of~\eqref{slcs}. Note that in~\cite{conner1984state} it is assumed that the matrices $A_i$ are invertible and that all the matrices $B_i$ have the same rank. Related results are provided in~\cite{conner1987structure}.  The problem is also considered in the short note~\cite{egerstedt2005observability} where the authors give the upper bound $n^2$ for the minimal length of controllable sequences in the particular case in which $A_i=A$ for every $i$, with $A$ possibly singular.
The minimal number of time steps for controllability is also estimated in~\cite{ji2007constructive} by explicit construction of a controllable sequence (see also \cite{liu2017reachability} for an analogous construction), with preliminary results in \cite{ge2001reachability} where it is shown that, if the matrices $A_i$ are invertible, the reachable space is given by $\sum_{i_1,\dots,i_n=1,\dots,m}^{j_1,\dots,j_n=0,\dots,n-1}A_{i_n}^{j_n}\cdots A_{i_1}^{j_1}\mathrm{Im}(B_{i_1})$, and controllability conditions are derived. 
%
We note that a similar problem has been studied in~\cite{yang2002algebraic,sun2002controllability} for  continuous-time switched  linear control systems (with some preliminary result in~\cite{sun2001reachability}). 
The more recent works \cite{zhu2021stabilizing} and \cite{zhu2024minimum} deal with the minimum number of switchings that a controllable sequence must have, both in discrete-time and continuous-time settings; in the first work the authors prove that such a number is always smaller than or equal to $\frac{n(n+1)}2-1$, while in the second work the latter bound is shown to be tight. 

Although in this work we focus on controllable sequences, the results presented in this paper can be readily transposed to observability problems using the well-known duality between controllability and observability (see e.g.~\cite{ge2001reachability}). Actually,  
this paper is in part motivated by our recent work~\cite{aazan:hal-04503485}, where an observer of a switched linear system is designed assuming the existence of observable switching sequences of a given length.

The organization of this paper and its main contributions are as follows. In Section~\ref{sec:minlen}, we derive an
upper bound on the length of the shortest controllable sequences. We show a family of systems where this bound is actually attained, showing its tightness. To prove these results, we first need to establish a novel characterization of the reachable space of a switched linear control system, which differs from that presented in~\cite{ge2001reachability}.
While in Section~\ref{sec:minlen} we present bounds that are independent of the rank of the matrices $B_i$, we provide rank-dependent bounds in Section~\ref{sec:rank}. A standing assumption in Sections~\ref{sec:minlen} and 
~\ref{sec:rank} is that all matrices $A_i$ are invertible. In Section~\ref{sec:degenerate}, we provide some discussion for the case of systems where this assumption does not hold. 
Concluding remarks can be found in Section~\ref{sec:conclusion}.

{\bf Notations:} The set of modes is denoted by $\Sigma=\{1,\dots,m\}$. The set $\Sigma^k$ indicates all sequences of $k$ values of $\Sigma$, we write them as $i_1\cdots i_k$ and we say that they have length $k$. We set $\Sigma^* = \cup_{k\in\mathbb{N}} \Sigma^k$, where $\Sigma^0 = \{\epsilon\}$, $\epsilon$ being the empty sequence. Every sequence in $\Sigma^k$ can be identified with a ``switching'' sequence for \eqref{slcs}. In particular given $\pi = i_1\cdots i_k\in \Sigma^k$ we denote by $A_{\pi}$ the evolution matrix at time $k$ for \eqref{slcs} with zero inputs, that is $A_{\pi} = A_{i_{k}}\cdots A_{i_1}$. 
Furthermore, given two sequences $\pi_1 = i_1\cdots i_{k_1}$ and $\pi_2 = j_1\cdots j_{k_2}$ we denote by $\pi_1\pi_2$ the concatenation $i_1\cdots i_{k_1} j_1 \cdots j_{k_2}$.
The image of a linear application corresponding to a matrix $A$ is denoted as $\mathrm{Im}(A)$. Given a $n\times n$ matrix $A$ and a vector space $V\subseteq \mathbb{R}^n$ we denote by $AV$ the image of the restriction of $A$  to $V$.
Given a set $S\subseteq \mathbb{R}^n$, the span of $S$ denoted $\mathrm{span}(S)$ is the smallest subspace of $\mathbb{R}^n$ containing $S$.
Let $\mathcal{I}$ be a finite index set and $V_i$ a vector subspace of $\mathbb{R}^n$ for $i\in \mathcal{I}$, then the sum of spaces $\sum_{i\in\mathcal{I}}V_i$ is the smallest subspace of $\mathbb{R}^n$ containing $V_i$ for every $i\in\mathcal{I}$, i.e. $\sum_{i\in\mathcal{I}}V_i =
\mathrm{span}(\bigcup_{i\in\mathcal{I}} V_i)$.
We indicate by $e_1,\dots,e_k$ the elements of the canonical basis of $\mathbb{R}^k$ (where $k$ is clear from the context).
In the following $\mathbb{N}$ is the set of natural numbers and $\mathbb{N}^*=\mathbb{N}\setminus\{0\}$.

\section{Controllable sequences of minimal length}
\label{sec:minlen}


We consider here switched linear control systems of the form~\eqref{slcs} assuming that the matrices $A_i$ are invertible and that not all the matrices $B_i$ are zero. We will not impose any condition on the dimension (number of columns) of the matrices $B_i$.

Given a sequence $\pi\in \Sigma^k$ for some nonnegative integer $k$, we denote by $\mathcal{R}(\pi)$ the \emph{reachable space} of \eqref{slcs} from the origin corresponding to $\pi$ (with the convention $\mathcal{R}(\epsilon)=\{0\}$). This set corresponds to the points that can be reached at time $k$ from the origin with the dynamics~\eqref{slcs} and the sequence $\pi$, for all possible control inputs $u_1,\dots,u_k$. By linearity with respect to the control inputs, $\mathcal{R}(\pi)$ is a vector space, and it is well-known that, for $\pi=i_1\cdots i_k$, it can be characterized as
\[\mathcal{R}(\pi) = \mathrm{Im}(A_{i_k}\cdots A_{i_{2}}B_{i_{1}},\ldots ,A_{i_{k}}B_{i_{k-1}},B_{i_k}).\] 
The \emph{reachable set} $\mathcal{R}$ of~\eqref{slcs} (from the origin) is the union of the reachable spaces for all possible sequences, that is, 
\begin{equation}
\label{eq:reach}
\mathcal{R}= \bigcup_{\pi\in \Sigma^*} \mathcal{R}(\pi).
\end{equation}
We will say that~\eqref{slcs} is \emph{controllable} if, for all $x\in\mathbb{R}^n$, there exists $k\in \mathbb{N}$, a switching sequence $ i_1\cdots i_k$  and control inputs $u_1,\dots, u_k$ such that $x_0=0$ and $x_k=x$, i.e., if $\mathcal{R} = \mathbb{R}^n$. Furthermore, we say that $\pi\in\Sigma^*$ is a \emph{controllable sequence} if $\mathcal{R}(\pi)=\mathbb{R}^n$.

In this section, we establish a tight upper bound on the length of the shortest controllable sequence. To derive such a bound, we first give a novel characterization of the reachable set of~\eqref{slcs}.

\subsection{Characterization of the reachable set}

Let us introduce the sequence of subspaces defined recursively for $k\in \mathbb{N}^*$ by
\begin{align}
    \label{eq:V1}
    V_1 &= \sum_{i=1}^m \mathrm{Im}(B_i), \\
    \label{eq:V2}
    V_{k+1}& = V_k + \sum_{i=1}^m A_i V_k.
\end{align}
The following result shows that the sequence reaches a fixed point in number of iterations bounded by $n$.
\smallskip

\begin{lemma} 
\label{nested}
There exists $\ell\leq n$ such that $V_k=V_{\ell}$ for every $k>\ell$.
\end{lemma}

\begin{proof}
Let ${\ell}$ be the smallest positive integer such that $V_{\ell+1}=V_{\ell}$.
Since $V_1$ is a nonzero subspace of $\mathbb{R}^n$ and $V_{k}\subsetneq V_{k+1}$ if $1\leq k<\ell$ it follows that the dimension of $V_{\ell}$ is at least $\ell$ and therefore $\ell\leq n$. The fact that $V_k=V_{\ell}$ for every $k>\ell$ follows from the definition of the spaces~$V_k$.
\end{proof}

We can now provide a characterization of the reachable set $\mathcal{R}$ of \eqref{slcs}. 
\smallskip

\begin{theorem}
\label{main++}
Consider system~\eqref{slcs} and assume that the matrices $A_i$ are invertible for $i=1,\dots,m$.
Let $V_\ell$ be the fixed point of the sequence defined by~\eqref{eq:V1}-\eqref{eq:V2}.
Then,
\begin{align}
\mathcal{R}  &=V_\ell\nonumber\\
& = \mathrm{span}\{A_{i_{k}}\cdots A_{i_2}b_{i_1}\mid  i_1\cdots i_{k}\in \Sigma^{k},\ 
1\leq k\leq n,\ b_{i_{1}}\in \mathrm{Im}(B_{i_{1}})\}.
\label{eq:th1}
\end{align}
In particular,~\eqref{slcs} is controllable if and only if
\[\mathrm{span}\{A_{i_{k}}\cdots  A_{i_2}b_{i_1}\mid i_1\cdots i_{k}\in \Sigma^{k},\ 1\leq k\leq n,\ b_{i_{1}}\in \mathrm{Im}(B_{i_{1}})\} = \mathbb{R}^n.
\]
\end{theorem}

The proof of Theorem~\ref{main++} is based on two lemmas that are introduced next.
We note that for a sequence $\pi = i_1\cdots i_k$ 
the set $\mathcal{R}(\pi)$ is formed by all elements of $\mathbb{R}^n$ of the form 
\begin{equation}\label{reachable-vectors}
A_{i_{k}}\cdots A_{i_2}B_{i_1}u_{1}+A_{i_{k}}\cdots A_{i_3}B_{i_2}u_{2}+\dots + B_{i_k} u_{k},
\end{equation}
 with $u_1,\dots,u_k$ vectors of appropriate dimensions. 
\smallskip

\begin{lemma}
\label{concatenate}
We have
\[\mathcal{R}(\pi_1\pi_2)= A_{\pi_2} \mathcal{R}(\pi_1) + \mathcal{R}(\pi_2)\]
for every sequences $\pi_1,\pi_2$.
\end{lemma}
\begin{proof}
By~\eqref{reachable-vectors}, if $\pi_1 = i_1\cdots i_{k_1}$ and $\pi_2=j_1\cdots j_{k_2}$, $\mathcal{R}(\pi_1\pi_2)$ coincides with the space of elements of $\mathbb{R}^n$ which can be written as
\begin{align*}
 A&_{j_{k_2}}\cdots A_{i_2}B_{i_1}u_{1}+\dots \ldots + A_{j_{k_2}}\cdots A_{j_1}B_{i_{k_1}}u_{{k_1}} \\
&+ A_{j_{k_2}}\cdots A_{j_2}B_{j_1}v_{1} +\ldots  + B_{j_{k_2}}v_{{k_2}}\\
= & A_{\pi_2} (A_{i_{k_1}}\cdots A_{i_2}B_{i_1} u_{1}+\ldots+B_{i_{k_1}}u_{{k_1}})\\
& + A_{j_{k_2}}\cdots A_{j_2}B_{j_1}v_{1} +\ldots+ B_{j_{k_2}}v_{{k_2}},
\end{align*}
for some vectors $u_1,\dots,u_{k_1},v_1,\dots,v_{k_2}$. The lemma follows as an immediate consequence.
\end{proof}

\begin{lemma}
\label{link}
We have for all $p\ge 1$
\begin{equation}
V_p = \mathrm{span}\{A_{i_{k}}\cdots A_{i_2}b_{i_1}\mid i_1\cdots i_{k}\in \Sigma^{k},\ 1\leq k\leq p,\ b_{i_{1}}\in \mathrm{Im}(B_{i_{1}})\}
\label{eq:rewrite}
\end{equation}
where $V_p$ is as in \eqref{eq:V1}-\eqref{eq:V2}. Furthermore
\begin{equation}
\label{vk}
V_p = \sum_{\pi\in\cup_{k=1}^p\Sigma^k} \mathcal{R}(\pi).
\end{equation}
\end{lemma}
\begin{proof}
Equation~\eqref{eq:rewrite} simply follows from the recursive construction of $V_p$. Concerning~\eqref{vk}, the inclusion $\supseteq$ follows from the explicit expression~\eqref{reachable-vectors} of the elements of $\mathcal{R}(\pi)$ for a sequence $\pi$. The inclusion $\subseteq$ is obtained by observing that $V_p$ is generated by elements  of the form  $A_{i_{k}}\cdots A_{i_2}b_{i_1}$, with $k\leq p$, where $b_{i_1}$ is a column of $B_{i_1}$; by~\eqref{reachable-vectors}  each element of this form belongs to $\mathcal{R}(\bar\pi)$ where $\bar \pi = i_1\cdots i_{k}\in\Sigma^k$ and corresponds to the sequence of control  inputs $u_1=e_j,~u_2=\dots=u_{k}=0$ for some index $j$.
\end{proof}

We can now prove the theorem.
\begin{proof}[Proof of Theorem~\ref{main++}]
We note that from Lemma~\ref{nested} it follows that $V_{\ell} = V_n$. This, together with~\eqref{eq:rewrite}, gives the 
second equality in \eqref{eq:th1}.
Moreover, it follows from~\eqref{vk} and from Lemma~\ref{nested}  that
$$
\mathrm{span}(\mathcal R) = \sum_{\pi\in\Sigma^*} \mathcal{R}(\pi)=\sum_{k\in \mathbb{N}^*} V_k =V_\ell.
$$
Hence,
the first equality in \eqref{eq:th1} holds true when we replace $\mathcal{R}$ by $\mathrm{span}(\mathcal{R})$. It then remains to show that $\mathcal{R}$ is a vector space\footnote{The fact that $\mathcal{R}$ is a vector space when the matrices $A_i$ are invertible was actually already shown in~\cite{stanford1980controllability}. 
We include a proof here for the sake of completeness.}
For this purpose consider a vector space $\mathcal{R}(\pi)$ having maximal dimension among all sequences $\pi\in\Sigma^*$. By Lemma~\ref{concatenate} we have $\mathcal{R}(\pi_1 \pi) = A_{\pi} \mathcal{R}(\pi_1) + \mathcal{R}(\pi)$ for every $\pi_1 \in\Sigma^*$, and in particular $\mathcal{R}(\pi_1 \pi)$ contains $\mathcal{R}(\pi)$. By maximality of $\mathcal{R}(\pi)$ we then have  $\mathcal{R}(\pi_1 \pi)=\mathcal{R}(\pi)$. It follows that $A_{\pi} \mathcal{R}(\pi_1) \subseteq \mathcal{R}(\pi)$ and, since $A_{\pi}$ is nonsingular, $\mathcal{R}(\pi_1) \subseteq A_{\pi}^{-1}\mathcal{R}(\pi)$ for every $\pi_1\in\Sigma^*$. In the special case $\pi_1=\pi$ it follows that $\mathcal{R}(\pi) = A_{\pi}^{-1}\mathcal{R}(\pi)$ 
so that $\mathcal{R}(\pi_1)\subseteq A_{\pi}^{-1}\mathcal{R}(\pi) = \mathcal{R}(\pi)$ for every sequence $\pi_1$. This implies that $\mathcal{R}(\pi) = \mathcal{R}$.
\end{proof}

Theorem~\ref{main++} provides a novel characterization of the reachable set of discrete-time switched linear control systems\footnote{An analogous characterization of the reachable set of continuous-time switched linear control systems has been recently pointed out in \cite[Theorem 7.7]{book-switched}.}.  
In~\cite{ge2001reachability}, it was shown that
$$
\mathcal{R}=\sum_{i_1,\dots,i_n=1,\dots,m}^{j_1,\dots,j_n=0,\dots,n-1}A_{i_n}^{j_n}\cdots A_{i_1}^{j_1}\mathrm{Im}(B_{i_1}).
$$
In comparison, the characterization given by Theorem~\ref{main++} reads:
$$
\mathcal{R}=\sum_{i_1,\dots,i_n=1,\dots,m}^{j_2,\dots,j_n=0,1}A_{i_n}^{j_n}\cdots A_{i_2}^{j_2}\mathrm{Im}(B_{i_1}).
$$
It is readily seen that
the latter expression involves
sums with fewer terms ($m^n \times 2^{n-1}$ instead of
$m^n \times n^{n}$), each of which involve
less products of matrices (up to $n$ instead of $n^n$). 
Moreover, \eqref{eq:V1}-\eqref{eq:V2} provides an effective algorithm for computing the reachable space.
Finally, for $m=1$, \eqref{eq:th1} reduces to the classical characterization of the reachable set from the origin for linear control systems and to the classical Kalman controllability criterion. Hence, it represents a very natural generalization of these results in the setting of switched linear control systems. 
This characterization will be instrumental in deriving upper bounds on the length of the shortest controllable sequence, presented next.

\subsection{Minimal length of controllable sequences}

We can now establish the following tight upper bound on the length of controllable sequences of minimal length:
\smallskip

\begin{theorem}
\label{main++2}
Consider the system~\eqref{slcs} and assume that the matrices $A_i$ are invertible for $i=1,\dots,m$.
The following results hold true:
\begin{itemize}
\item[1.] If~\eqref{slcs} is controllable, then there exists a controllable sequence $\pi$ of length $k\leq \frac{n(n+1)}{2}$.
\item[2.] There exist pairs $(A_i,B_i)$ for $i=1,\dots,m$ such that~\eqref{slcs} is controllable and the shortest controllable sequences have length 
\begin{itemize}
\item[(a)] $n+\frac{m(m-1)}{2}$, if $m \leq n$, 
\item[(b)] $\frac{n(n+1)}{2}$, if $m > n$.
\end{itemize}
\end{itemize}
\end{theorem}

Concerning the second item, it essentially states that the upper bound on the length of the shortest controllable sequences obtained in the first item is actually optimal when $m\ge n$. This result may appear counterintuitive, since one may think that, when the shortest controllable sequence is applied, the dimension of the space reached from the origin strictly increases at each time step, implying that the length of the sequence does not exceed $n$. The fact that this is not always the case has already been observed in~\cite{conner1984state}.

The proof of Theorem~\ref{main++2} relies on the following lemma.
\smallskip

\begin{lemma}
\label{concatenate2}
Assume that the matrices $A_i$ are invertible for $i=1,\dots,m$.
Consider the spaces $V_k$ in \eqref{eq:V1}-\eqref{eq:V2} and set $V_0 = \{0\}$.
Take a sequence $\pi\in\Sigma^*$ and let $k\geq 1$ be such that \[\mathrm{dim}(V_{k-1})\leq\mathrm{dim}(\mathcal{R}(\pi))<\mathrm{dim}(V_{k}).\]
Then there exists a sequence $\pi_0$ of length no greater than $k$ such that \[\mathrm{dim}(\mathcal{R}(\pi_0\pi))>\mathrm{dim}(\mathcal{R}(\pi)).\]
\end{lemma}
\begin{proof}
Using Lemma~\ref{concatenate} we have $\mathcal{R}(\pi_0\pi)= A_{\pi} \mathcal{R}(\pi_0) + \mathcal{R}(\pi).$ Since $A_{\pi}$ is an invertible matrix, we can consider the space $A_{\pi}^{-1} \mathcal{R}(\pi)$, which has the same dimension as $\mathcal{R}(\pi)$. 
As $\mathrm{dim}(\mathcal{R}(\pi))<\mathrm{dim}(V_{k})$ we then deduce from~\eqref{vk} the existence of $\pi_0\in\cup_{h=1}^k\Sigma^h$ 
such that $\mathcal{R}(\pi_0)\not\subseteq A_{\pi}^{-1} \mathcal{R}(\pi)$ and therefore \[\mathrm{dim}(\mathcal{R}(\pi_0\pi))= \mathrm{dim}(A_{\pi} \mathcal{R}(\pi_0) + \mathcal{R}(\pi))>\mathrm{dim}(\mathcal{R}(\pi)),\] concluding the proof of the lemma.
\end{proof}

\begin{proof}[Proof of Theorem~\ref{main++2}]
Let us  prove Item \textit{1.} By applying Lemma~\ref{concatenate2} iteratively we can concatenate at most $\mathrm{dim}(V_1)$ sequences of length one (i.e., symbols in $\Sigma$) to obtain a sequence $\pi_1$ such that $\mathrm{dim}(\mathcal{R}(\pi_1))  \geq \dim(V_1)$. Similarly we may find a concatenation $\pi_2$ of at most  $\mathrm{dim}(V_2)-\mathrm{dim}(V_1)$ sequences of length no greater than two such that $\mathrm{dim}(\mathcal{R}(\pi_2 \pi_1))  \geq \dim(V_2)$. More in general, for $k>2$ we construct recursively $\pi_k$ of total length bounded by $k \big(\mathrm{dim}(V_k)-\mathrm{dim}(V_{k-1})\big)$ such that $\mathrm{dim}(\mathcal{R}(\pi_k \cdots \pi_1))  \geq \dim(V_k)$. Hence if $\ell$ is as in Lemma~\ref{nested}, that is if $V_\ell =\mathbb{R}^n$, we have $\mathcal{R}(\pi_\ell \cdots \pi_1)=\mathbb{R}^n$, i.e., $\pi=\pi_\ell \cdots \pi_1$ is a controllable sequence. Then the total length of $\pi$ is bounded by 
\begin{align*}
\sum_{k= 1}^\ell k & \big(\mathrm{dim}(V_k)-\mathrm{dim}(V_{k-1})\big)\\  &= \sum_{h=0}^{\ell-1} \big(\mathrm{dim}(V_\ell)-\mathrm{dim}(V_{h})\big)\\
& \leq n+(n-1)+\dots+(n-\ell+1) \\
&\leq  \frac{n(n+1)}{2},
\end{align*}
with equality holding if $\mathrm{dim}(V_k) = k$ for $k\leq \ell=n$. This concludes the proof of Item \textit{1.}

{We are left to prove Item \textit{2.}
Let us first consider case $(a)$. The case $m=1$ being trivial, we construct, inspired by the work in~\cite{zhu2024minimum}, for any given $n\geq m \geq 2$, an example of controllable system of the form~\eqref{slcs} such that the corresponding controllable sequences satisfy the required lower bound.
For $k\in\mathbb{N}^*$, let us consider the $k\times k$ matrix
\begin{equation}\label{Pk}
P_k=\begin{pmatrix} 
 0  & 1\\
\mathrm{Id}_{k-1} & 0
 \end{pmatrix}
\end{equation}
with the convention $P_1=1$.

Define the matrices $(A_k,B_k)$ as
\begin{align*}A_1&=\begin{pmatrix} 
P_{n-m+1} & 0\\
0 & \mathrm{Id}_{m-1}
 \end{pmatrix},\qquad B_1=e_1\\
 A_k&=\begin{pmatrix} 
\mathrm{Id}_{n-m+k-2} & 0 & 0\\
0 & P_2 & 0\\
0 & 0 & \mathrm{Id}_{m-k}
 \end{pmatrix},\qquad B_k=0,
 \end{align*}
for $k=2,\dots,m$.
In particular, applying the first mode starting from the space span$\{e_{1},\dots,e_{j}\}$ allows us to reach span$\{e_1,\dots,e_{j+1}\}$ if $j\leq n-m$.
The other modes, when applied to a space of the form $ \mathrm{span}\{e_{i_1},\dots,e_{i_h}\}$ for some indices $i_1,\dots,i_h$, either let the space unchanged or they  raise or lower one of the indices by one unit.
We construct a controllable sequence as follows.
We first apply $n-m+1$ times the mode 1. Starting from the origin this allows us to generate the space span$\{e_{1},\dots,e_{n-m+1}\}$, i.e., indicating by $\pi_1$ the corresponding word we get $\mathcal{R}(\pi_1) = \mathrm{span}\{e_{1},\dots,e_{n-m+1}\}$. We then apply sequentially the modes $2,3,\dots,m$, i.e., the sequence $\pi_2 = 23\cdots m$, to obtain the space $\mathcal{R}(\pi_1\pi_2) = \mathrm{span}\{e_1,\dots,e_{n-m},e_n\}$. Next, considering the sequence $\pi_3=12\cdots (m-1)$ one has $\mathcal{R}(\pi_1\pi_2\pi_3) = \mathrm{span}\{e_1,\dots,e_{n-m},e_{n-1},e_n\}$. More in general, defining recursively $\pi_k = 12\cdots (m+2-k)$ for $k\leq m+1$ we obtain  $\mathcal{R}(\pi_1\cdots\pi_k) = \mathrm{span}\{e_1,\dots,e_{n-m},e_{n-k+2},\dots,e_n\}$. In particular $\mathcal{R}(\pi_1\cdots\pi_{m+1}) = \mathbb{R}^n$, that is, $\pi := \pi_1\cdots\pi_{m+1}$ is a controllable sequence. The total length of the sequence $\pi$ is 
\begin{align*}
(n&-m+1)+(m-1)+\sum_{k=3}^{m+1}(m+2-k)\\
&= n+\frac{m(m-1)}{2}.
\end{align*}
To prove that this represents the minimal length for controllable sequences for the given switched linear control system, we associate with the $i$-th element of the basis the weight 
 \begin{equation*}
 \nu_i=
 \begin{cases}
 1&\text{ if }i=1,\dots,n-m+1,\\
 i+m-n&\text{ if }i=n-m+2,\dots,n
 \end{cases}
 \end{equation*}
We then define \[\nu( \mathrm{span}\{e_{i_1},\dots,e_{i_h}\}) = \sum_{j=1}^h \nu_{i_j}.\]
It is easy to check that, for every $k=1,\dots,m$, the value $\nu$ increases at most by one unit when one applies mode $k$ to a space  $ \mathrm{span}\{e_{i_1},\dots,e_{i_h}\}$ for every (possibly empty) set of distinct indices $\{i_1,\dots,i_h\}\subseteq\{1,\dots,n\}$.
Hence to reach $\mathbb{R}^n$ from the origin we need to apply sequentially at least $\nu(\mathbb{R}^n)$ modes, and we have
\begin{align*}
\nu(\mathbb{R}^n) &= n-m + \sum_{i=1}^m i \\
& = n +\frac{m(m-1)}{2}. 
\end{align*}
We have then obtained the shortest length of controllable sequences, which concludes the proof of case $(a)$.
}

It remains to prove case $(b)$.
It should be noted that the sought lower bound on the maximal length of controllable sequences is a nondecreasing function of the number of modes $m$ (assuming the dimension $n$ fixed) since, for instance, duplicating a mode does not change the maximal length of controllable sequences. 
Then, case $(b)$ follows from case $(a)$ with $m=n$.
\end{proof}

Item \textit{1.} of Theorem~\ref{main++2} provides an upper bound on the length of the shortest controllable sequence that is independent of the number of modes $m$.  
Let us mention that the proof describes an effective method to construct a controllable sequence whose length is smaller than $\frac{(n+1)n}{2}$.
It then follows from Item \textit{2.} that this bound is actually optimal when $m\ge n$. 
When $m<n$, the upper bound given by Item \textit{1.} is probably not tight, as evidenced by the discrepancy with the lower bound provided by Item \textit{2.} 
It is clearly not tight for $m=1$.
Note that the lower bound provided by Item \textit{2.} is not tight when $m<n$ as shown in the following example.
\smallskip

\begin{figure*}
\begin{center}
    \includegraphics[width=0.9\textwidth]{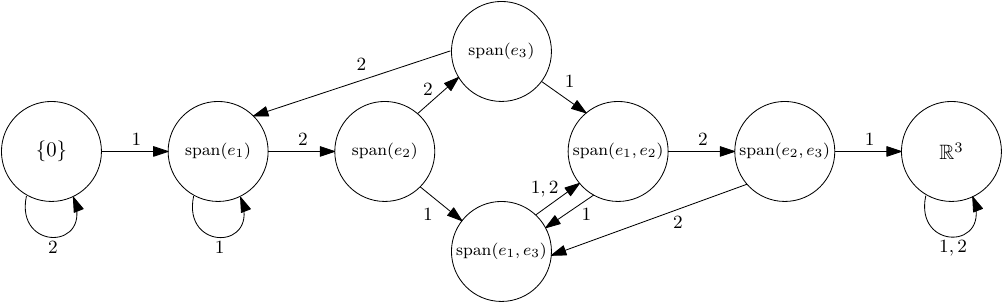}
    \caption{Automaton describing the evolution of the reachable space under all possible switching signals for system in Example~\ref{example}.\label{fig}}
\end{center}
\end{figure*}

\begin{example}\label{example}
We consider the system with $n=3$ and $m=2$, given by the matrices:
$$
A_1=\left(\begin{matrix}
    1 & 0 & 0 \\
    0 & 0 & 1 \\
    0 & 1 & 0
\end{matrix}\right),\;
A_2=\left(\begin{matrix}
    0 & 0 & 1 \\
    1 & 0 & 0 \\
    0 & 1 & 0
\end{matrix}\right),\;
B_1=\left(\begin{matrix}
    1  \\
    0  \\
    0 
\end{matrix}\right),\;
B_2=\left(\begin{matrix}
    0  \\
    0  \\
    0 
\end{matrix}\right).
$$
This system have the following controllable sequences $122121$, $121121$, $121221$ of length $6$, and no controllable sequence of length $5$ or less. This can be checked using the automaton in Figure~\ref{fig}, which describes the evolution of the reachable space under all possible switching signals. The shortest paths in the automaton from the null space to $\mathbb R^3$ correspond to the shortest controllable sequences.
Let us remark that their length  corresponds to the upper bound  $\frac{n(n+1)}{2}=6$ while being strictly greater than the lower bound $n+\frac{m(m-1)}{2}=4$.
\end{example}
\medskip

\begin{remark}
\label{rem-equivalent}
The assumption that all matrices $A_i$ are invertible may be replaced by the relaxed assumption that each $(A_i,B_i)$ has maximal rank, i.e., $\mathrm{Im}(A_i)+\mathrm{Im}(B_i) = \mathbb{R}^n$. Indeed, in the latter case one can consider a switched linear control system corresponding to pairs $(\hat A_i,B_i)$ where $\hat A_i = A_i+B_i K_i$, with $K_i$ matrix gains of appropriate dimensions. Up to appropriately choosing $K_i$, one has that $\hat A_i$ is invertible and, furthermore, the modified system gives rise to the same trajectories as the initial one (with the control inputs $u_k$ replaced by $\hat u_k=u_k-K_{i_k}x_{k-1}$). Therefore Theorems~\ref{main++} and ~\ref{main++2} apply to the modified system and hence to the initial system with the relaxed assumption.

On the other hand a straightforward necessary condition for~\eqref{slcs} to admit a controllable sequence is that there exists at least one mode (the last one in  controllable sequences) satisfying $\mathrm{Im}(A_i)+\mathrm{Im}(B_i) = \mathbb{R}^n$. 
Note that the fact, shown in Item \textit{1.} of Theorem~\ref{main++2} for invertible matrices $A_i$, that every controllable system~\eqref{slcs} admits a controllable sequence holds true even in the presence of degenerate matrices $A_i$. This result, already shown in~\cite{stanford1980controllability}, can be obtained by observing that $\mathcal{R}$ is a countable union of subspaces $\mathcal{R}(\pi)$ of $\mathbb{R}^n$, for $\pi\in\Sigma^*$, hence it is a measurable set of Lebesgue measure zero unless one of such subspaces is equal to $\mathbb{R}^n$.
\end{remark} 

\section{Rank-dependent bounds}
\label{sec:rank}

We now focus on the problem of estimating the minimal length of controllable sequences following the point of view of~\cite{conner1984state,conner1987structure}. In those papers, the authors introduce an index $\mathcal B(n,r,m)$  which maximizes the minimal length of controllable sequences among all switched linear control systems evolving on $\mathbb{R}^n$, consisting of $m$ modes, such that every $A_i$ is invertible and every $B_i$ has rank $r$.
They obtain the exact expression or some bounds of $\mathcal B(n,r,m)$  in some special cases, namely in \cite{conner1984state} they prove that \[\mathcal B(n,n-1,m)=2\ \mbox{ and }\ \mathcal B(3,1,m) = 4\ \mbox{ if }\ m\geq 2,\] while in~\cite{conner1987structure} they show that \[\mathcal B(n,r,m) \geq 2n-2r\ \mbox{ if }\ m\geq 2.\]
{In the following we will also consider the slightly different quantity $\bar{\mathcal B}(n,r,m)$ where the maximization is done over all switched linear control systems evolving on $\mathbb{R}^n$, consisting of $m$ modes, such that every $A_i$ is invertible and every $B_i$ has rank \emph{at least} $r$. Note that estimates of $\bar{\mathcal B}(n,0,m)$ are obtained in Theorem~\ref{main++2} for every $n,m$, while it does not make sense to consider $\mathcal B(n,0,m)$ for lack of controllability of the corresponding switched systems.
A general result is given below. 
\smallskip

\begin{theorem}
\label{th-conner}
Let $r\geq 1$. We have 
\begin{equation}\label{eq-bnrm-1}
\mathcal B(n,r,m) \leq \bar{\mathcal B}(n,r,m) \leq \frac{(n-r+1)(n-r)}{2}+1
\end{equation}
for every $m$ and $r<n$.
Furthermore, if $m\leq n-r$ the following inequalities hold
\begin{equation}\label{eq-bnrm-2}
n-r+1+\frac{m(m-1)}{2} \leq \mathcal B(n,r,m) \leq \bar{\mathcal B}(n,r,m).
\end{equation}
As a consequence, equalities hold in \eqref{eq-bnrm-1} if  $m\geq n-r$.
\end{theorem}
}

In order to prove Theorem~\ref{th-conner}, we will need the following preliminary result.
\smallskip

\begin{lemma}
\label{l-transform}
Assume that $M$ is a $n\times n$ invertible matrix. Then, defining  $\mathcal{M}$ as the set of  all matrices $P\in\mathbb{R}^{r\times (n-r)}$ such that 
\[\begin{pmatrix}\mathrm{Id}_{r} & P\\ 0 & \mathrm{Id}_{n-r}\end{pmatrix}^{-1}M \begin{pmatrix}\mathrm{Id}_{r} & P\\ 0 & \mathrm{Id}_{n-r}\end{pmatrix} = \begin{pmatrix} * & *\\ * & \hat{M} \end{pmatrix}\]
for some invertible $(n-r)\times (n-r)$ matrix $\hat{M}$, we have that $\mathcal{M}$
is open and dense in $\mathbb{R}^{r\times (n-r)}$.
\end{lemma}
\begin{proof}
Writing $M=\left(\begin{smallmatrix} A & B\\ C & D\end{smallmatrix}\right)$ where $A\in\mathbb{R}^{r\times r}$, $B\in\mathbb{R}^{r\times (n-r)}$, $C\in\mathbb{R}^{(n-r)\times r}$, $D\in\mathbb{R}^{(n-r)\times (n-r)}$, we get, by a direct computation, $\hat M = CP+D$. 
We start by looking for a matrix $P$ such that $\hat M$ is invertible. If $D$ is invertible then one can choose $P=0$, hence we assume that the rank of $D$ is $n-r-p$ for some positive $p$. In particular there exists $i_1,\dots,i_p$ in $\{1,\dots,n-r\}$ such that the $i_1$-th,\dots,$i_p$-th columns of $D$ can be expressed as a linear combination of the remaining linearly independent $n-r-p$ columns of $D$. 
Since $M$ is invertible it follows that the $(n-r)\times n$ matrix $(C\ D)$ has rank $n-r$. In particular we can assume that the $j_1$-th,\dots,$j_p$-th columns of $C$ form, together with the $n-r-p$ already identified linearly independent columns of $D$, a basis of $\mathbb{R}^{n-r}$. Setting $\bar{P} = \sum_{k=1}^p e_{j_k}e_{i_k}^{\top}$ (here $e_{j_k},e_{i_k}$ are elements of the canonical bases of $\mathbb{R}^{r}$ and $\mathbb{R}^{n-r}$, respectively) it follows that adding $C\bar{P}$ to $D$ amounts to adding the $j_k$-th column of $C$ to the $i_k$-th column of $D$ for $k=1,\dots,p$, keeping the remaining columns of $D$ unchanged. The columns of the resulting sum are then linearly independent, proving the invertibility of $C\bar{P}+D$. We have proved that the set $\mathcal{M}$ is always nonempty.
We then observe that $\mathrm{det}(CP+D)$ is a nonzero polynomial function of the entries of $P\in\mathbb{R}^{r\times (n-r)}$. In particular, it is nonzero on an open and dense subset of $\mathbb{R}^{r\times (n-r)}$. The latter set coincides with $\mathcal{M}$, and this concludes the proof of the lemma.
\end{proof}
We can now prove the theorem.

\begin{proof}[Proof of Theorem~\ref{th-conner}]
Note that the inequality
\[\mathcal B(n,r,m)\leq \bar{\mathcal B}(n,r,m)\]
is always verified, since $\bar{\mathcal B}(n,r,m)$ is defined as a maximum over a larger family of controllable systems of the form \eqref{slcs}, compared to $\mathcal B(n,r,m)$.

Let us now prove~\eqref{eq-bnrm-1}.
Recall that, in order to estimate {$\bar{\mathcal B}(n,r,m)$}, we have to assume that the matrices $A_k$ are invertible and that {$\mathrm{rank}(B_k) \geq  r$} for every $k=1,\dots,m$. First, consider the case $\mathrm{dim}(V_1)>r$, where $V_1$ is as in \eqref{eq:V1}. 
Using recursively Lemma~\ref{concatenate2} and the fact that each $B_i$ has rank {at least} $r$ (that is, {$\mathrm{dim}(\mathcal{R}(i)) \geq r$} for every $i\in \{1,\dots,m\}$) it is possible to construct a sequence $\pi_0$ of length at most $\mathrm{dim}(V_1) - r+1$ such that $\mathrm{dim}(\mathcal{R}(\pi_0))\geq \mathrm{dim}(V_1)$. Following  
 the proof of Theorem~\ref{main++2}, it is possible to extend $\pi_0$ to a controllable sequence of length no larger than 
 \begin{align*}
 \mathrm{dim}&(V_1)  - r+1+\sum_{k = 2}^{\ell} k \big(\mathrm{dim}(V_k)-\mathrm{dim}(V_{k-1})\big)\\
 &= \mathrm{dim}(V_\ell) - r+1+\sum_{k = 2}^{\ell} \big(\mathrm{dim}(V_{\ell})-\mathrm{dim}(V_{k-1})\big)\\
& \leq n - r+1+\sum_{h = 1}^{n-\mathrm{dim}(V_1)} h\\
& \leq n - r+1+\sum_{h = 1}^{n-r-1} h\\
& = \frac{(n-r+1)(n-r)}{2} + 1.
\end{align*}

Let us consider now the case $\mathrm{dim}(V_1) = r$. Then, without loss of generality, we can assume that the last $n-r$ rows of $B_k$ are zero for every $k=1,\dots,m$, i.e., $B_k = \left(\begin{smallmatrix} B_k' \\ 0_{(n-r)\times r}\end{smallmatrix}\right)$,  $B'_k$ being an invertible $r\times r$ matrix.
{By Lemma~\ref{l-transform}, for $k=1,\dots,m$ there exists a set $\mathcal{M}_k$ open and dense in $\mathbb{R}^{r\times (n-r)}$ such that, for every $P\in \mathcal{M}_k$, the coordinate transformation $Q=\left(\begin{smallmatrix}\mathrm{Id}_r & P\\ 0 & \mathrm{Id}_{n-r}\end{smallmatrix}\right)$ is such that, writing
\[Q^{-1}A_kQ=\begin{pmatrix}\hat D_k & \hat C_k\\ \hat B_k & \hat A_k\end{pmatrix},\] 
the matrix $\hat A_k$ is invertible. As the intersection of the sets $\mathcal{M}_k$ is nonempty (it is open and dense in $\mathbb{R}^{r\times (n-r)}$), we can take $P\in \cap_{k=1}^m \mathcal{M}_k$ ensuring that, after applying the corresponding coordinate transformation, all bottom-right blocks $\hat A_k$ are invertible.}
We denote by 
$x=
\left(
\begin{smallmatrix}
    y \\ \hat x
\end{smallmatrix}\right)$
 the state expressed in this coordinate basis, with $y\in \mathbb{R}^{r}$ and $\hat x\in\mathbb{R}^{n-r}$. Then, the dynamics corresponding to a sequence $i_1\cdots i_{k}$ is given by
\begin{equation}\label{slcs-change}
\begin{split}
y_{j} &= \hat D_{i_j} y_{j-1}+ \hat C_{i_j}\hat x_{j-1} + B'_{i_j} u_j,\\
 \hat x_{j} &= \hat B_{i_j} y_{j-1} + \hat A_{i_j}\hat x_{j-1}.
\end{split}
\end{equation}
In other words the components $\hat x$ follow the dynamics of a switched linear control system  
\begin{equation}
\label{eq-reduced}
\hat x_{j} =  \hat A_{i_j}\hat x_{j-1} + \hat B_{i_j} \hat u_j,
\end{equation}
where $\hat u_j = y_{j-1}$ is zero if $j=1$ and it is equal to $\hat D_{i_{j-1}} y_{j-2}+ \hat C_{i_{j-1}}\hat x_{j-2} + B'_{i_j} u_{j-1}$ otherwise. {We have that the reduced system~\eqref{eq-reduced} is controllable since~\eqref{slcs-change} is. Also, note that for every choice of $\hat u_j\in\mathbb{R}^r$ for $j>1$ there corresponds a unique choice $u_{j-1} =(B'_{i_j})^{-1}(\hat u_j - \hat D_{i_{j-1}} y_{j-2} - \hat C_{i_{j-1}}\hat x_{j-2})$ for~\eqref{slcs-change}, hence any arbitrary choice of the control sequence $\{\hat u_j\}_{j\geq 2}$ in \eqref{eq-reduced} gives rise to a trajectory of~\eqref{slcs-change}.}
Since $\hat u_1= 0$ in~\eqref{eq-reduced}, so that $\hat x_1=0$, it follows from Theorem~\ref{main++2} the existence of a controllable sequence for \eqref{eq-reduced} of length $\bar k\leq \frac{(n-r+1)(n-r)}{2}+1$. 
{This is equivalent to say that we can find a switching sequence of length $\bar k$ such that for every $\hat x\in \mathbb{R}^{n-r}$ there exists a state of the form $x=
\left(
\begin{smallmatrix}
    y \\ \hat x
\end{smallmatrix}\right)$
 reachable from the origin with the dynamics~\eqref{slcs-change} at time $\bar k$. As the choice of $u_{\bar k}$ does not affect the value $\hat x_{\bar k}$, by choosing $u_{\bar k} = (B'_{i_{\bar k}})^{-1} \left(y - \hat D_{i_j} y_{j-1} - \hat C_{i_j}\hat x_{j-1}\right)$ we can actually reach any point 
$
\left(
\begin{smallmatrix}
    y \\ \hat x
\end{smallmatrix}\right)\in \mathbb{R}^n$ showing that the considered switching sequence is controllable.
We have thus proved that the length of controllable sequences is bounded by $\frac{(n-r+1)(n-r)}{2}+1$ when $\mathrm{dim}(V_1)=r$, and this concludes the proof of~\eqref{eq-bnrm-1}.

In order to prove~\eqref{eq-bnrm-2} we show that the switched linear control system corresponding to the pairs $(A_k,B_k)$ where $B_k = \left(\begin{smallmatrix} \mathrm{Id}_r \\ 0_{(n-r)\times r}\end{smallmatrix}\right)$ for $k=1,\dots,m$ and
\begin{align*}
A_1&=\begin{pmatrix}\mathrm{Id}_{r-1} & 0 & 0\\ 0 & P_{n-r-m+2} & 0\\ 0 & 0 & \mathrm{Id}_{m-1}\end{pmatrix},\\
A_k&=\begin{pmatrix}\mathrm{Id}_{n-m+k-2} & 0 & 0\\ 0 & P_{2} & 0\\ 0 & 0 & \mathrm{Id}_{m-k}\end{pmatrix},\quad k=2,\dots,m, 
\end{align*}
with $P_k$ as in~\eqref{Pk}, admits controllable sequences of minimal length $n-r+1+\frac{m(m-1)}{2}$.
We start by constructing a controllable sequence of such length.
The first mode of the sequence may be chosen arbitrarily, since, independently of the choice, we reach the space $\mathrm{span}\{e_1,\dots,e_{r}\}$ starting from the origin. Then, applying $n-r-m+1$ times the first mode allows us to reach the space $\mathrm{span}\{e_1,\dots,e_{n-m+1}\}$. Applying sequentially the modes $2,3,\dots,m$ we end up with the space
$\mathrm{span}\{e_1,\dots,e_{n-m},e_n\}$. The controllable sequence can thus be constructed recursively by noting that, applying sequentially the modes $1,2,\dots,m-k$ starting from the space $\mathrm{span}\{e_1,\dots,e_{n-m},e_{n-k+1},\dots,e_n\}$, we reach the space $\mathrm{span}\{e_1,\dots,e_{n-m},e_{n-k},\dots,e_n\}$, for $k=1,\dots,m-1$.
The total length of the constructed controllable sequence is therefore
\begin{align*}
&1+(n-r-m+1)+(m-1)+\sum_{k=1}^{m-1}k\\
&=n-r+1+\frac{m(m-1)}{2}
\end{align*}
as required.

We next prove that the previously constructed controllable sequence is of minimal length.
For this purpose, we associate with the $i$-th element of the basis the weight 
 \begin{equation*}
 \nu_i=
 \begin{cases}
 0&\text{ if }i=1,\dots,r-1,\\
 1&\text{ if }i=r,\dots,n-m+1,\\
 i+m-n&\text{ if }i=n-m+2,\dots,n
 \end{cases}
 \end{equation*}
and define \[\nu( \mathrm{span}\{e_{i_1},\dots,e_{i_h}\}) = \sum_{j=1}^h \nu_{i_j}.\]
It is easy to check that, for every $k=1,\dots,m$, the value $\nu$ increases at most by one unit when one applies mode $k$ to a space  $ \mathrm{span}\{e_{i_1},\dots,e_{i_h}\}$ for every (possibly empty) set of distinct indices $\{i_1,\dots,i_h\}\subseteq\{1,\dots,n\}$.
Hence to reach $\mathbb{R}^n$ from the origin we need to apply sequentially at least $\nu(\mathbb{R}^n)$ modes, and we have
\begin{align*}
\nu(\mathbb{R}^n) &= n-m-r+2 + \sum_{i=n-m+2}^n (i+m-n) \\
& = n-r+1 +\frac{m(m-1)}{2}. 
\end{align*}
This shows that the previously constructed controllable sequence has minimal length, concluding the proof of~\eqref{eq-bnrm-2}.

As a consequence of~\eqref{eq-bnrm-2}, the inequalities in~\eqref{eq-bnrm-1} become equalities when $m=n-r$. The same holds for $m>n-r$ since, starting from a switched system with $n-r$ modes, by duplicating modes, we can create a switched system with $m>n-r$ modes preserving the original shortest controllable sequences.}
\end{proof}

\begin{remark}
As in Theorem~\ref{main++2}, the invertibility of the matrices $A_i$ assumed in Theorem~\ref{th-conner} can be replaced by the relaxed condition $\mathrm{Im}(A_i)+\mathrm{Im}(B_i) = \mathbb{R}^n$ for $i=1,\dots,m$.
\end{remark}

{
\section{The degenerate case}
\label{sec:degenerate}
In the previous sections we have assumed that all the matrices $A_i$ are invertible. In Remark~\ref{rem-equivalent} we have observed that this assumption can be relaxed by assuming that all modes satisfy the condition
\begin{equation}\label{relax}
\mathrm{Im}(A_i)+\mathrm{Im}(B_i) = \mathbb{R}^n,
\end{equation}
and that, in order to guarantee the controllability of the system or, equivalently, the existence of a controllable sequence, at least one mode must satisfy~\eqref{relax}. In the case in which the switched system is controllable but not all modes satisfy~\eqref{relax} (for simplicity, in the following we will refer to this situation as the \emph{degenerate case}), the proof scheme that we have adopted in this paper does not directly apply. In particular it is unclear if, in the degenerate case, one can find examples such that the upper bound $\frac{(n+1)n}{2}$ established in Theorem~\ref{main++2} on the length of shortest controllable sequences is violated. Also, to the best of the authors knowledge, an upper bound in the degenerate case has been previously obtained only in the special case in which $A_i = A$ for every mode $i$: in~\cite{egerstedt2005observability} it is shown that, under this assumption, controllable sequences of length no greater than $n^2$ exist for all controllable systems.
%
%
In general, for the degenerate case it is not even known if there exists an upper bound only depending on $n$ and $m$ on the length of the shortest controllable sequences.

While the question of a general upper bound remains open and, analogously, there is no proof that the bound \eqref{eq-bnrm-1} remains valid in the degenerate case, the lower bounds in Theorem~\ref{main++2} and Theorem~\ref{th-conner} can actually be improved by appropriately modifying the examples exhibited in the corresponding proofs.
Namely, one can modify  the example  exhibited in the proof of Theorem~\ref{main++2} by setting
\begin{equation}
\label{new-Ak}
A_k=\begin{pmatrix}0_{n-m+k-2} & 0 & 0\\ 0 & P_{2} & 0\\ 0 & 0 & \mathrm{Id}_{m-k}\end{pmatrix}, 
\end{equation}
 for $k=2,\dots,m$, keeping $A_1$ and the matrices $B_k$ unchanged. 
 Similarly to the aforementioned proof,  we can associate with the $i$-th element of the basis the weight 
 \begin{equation*}
 \nu_i=
 \begin{cases}
 1&\text{ if }i=1,\dots,n-m+1,\\
 i&\text{ if }i=n-m+2,\dots,n
 \end{cases}
 \end{equation*}
and consider the function 
\[\nu( \mathrm{span}\{e_{i_1},\dots,e_{i_h}\}) = \sum_{j=1}^h \nu_{i_j}.\]
It is easy to see that there exists a (unique) switching sequence along which the value $\nu$ increases by exactly one unit at each time step until the reached space coincides with $\mathbb{R}^n$. Hence the system is controllable and the minimal length of controllable sequences is no greater than
\begin{align*}
\nu( \mathbb{R}^n) &= n-m+1+\sum_{i=n-m+2}^n i\\
&=n-m+1+(n-m+1)(m-1)+\sum_{i=1}^{m-1} i\\
&=(n-m+1)m+\frac{m(m-1)}{2}.
\end{align*}
Furthermore, the value $\nu$ increases at most by one unit at each time step along any controllable sequence of minimal length. Indeed, it is easy to see that  $\nu(\mathcal{R}(\pi\sigma))\leq  \nu(\mathcal{R}(\pi))+1$ for every $\pi\in\Sigma^*$ and $\sigma\in \Sigma\setminus\{2\}$. If $e_{n-m+1}\notin\mathcal{R}(\pi)$ then $\nu(\mathcal{R}(\pi 2))\leq  \nu(\mathcal{R}(\pi))+1$, otherwise $\pi$ must be of the form $\pi=\bar\pi 1$ or $\pi=\bar\pi 2$ for some $\bar\pi\in\Sigma^*$. In the first case by an inductive argument one shows that $\mathrm{span}\{e_1,\dots,e_{n-m+1}\} \subseteq\mathcal{R}(\pi)$, so that $\nu(\mathcal{R}(\pi 2)) =  \nu(\mathcal{R}(\pi))+1$, while in the second case $\mathcal{R}(\pi 2) = \mathcal{R}(\bar \pi 2 2) \subseteq \mathcal{R}(\bar\pi)$, so that $\pi 2$  cannot be a subsequence of a controllable switching sequence of minimal length.
Hence, the minimal length of controllable sequences is exactly $(n-m+1)m+\frac{m(m-1)}{2}$. In the degenerate case, this strictly increases the lower bound $n+\frac{m(m-1)}{2}$ of Theorem~\ref{main++2} whenever $n>m>1$.

Similarly, concerning Theorem~\ref{th-conner}, one can consider a modification of the example  exhibited in its proof by replacing the matrices $A_k$ with the matrices in~\eqref{new-Ak} for $k=2,\dots,m$.
Proceeding as before and setting 
  \begin{equation*}
 \nu_i=
 \begin{cases}
 0&\text{ if }i=1,\dots,r-1,\\
1&\text{ if }i=r,\dots,n-m+1,\\
 i-r&\text{ if }i=n-m+2,\dots,n
 \end{cases}
 \end{equation*}
and 
\[\nu( \mathrm{span}\{e_{i_1},\dots,e_{i_h}\}) = \sum_{j=1}^h \nu_{i_j}\]
we can show the existence of controllable switching sequences of minimal length
\begin{align*}
\nu( &\mathbb{R}^n) = n-m-r+2+\sum_{i=n-m+2}^n (i-r)\\
&=n-m-r+2+(n-m-r+1)(m-1)+\sum_{i=1}^{m-1} i\\
&=\frac{m(2n-m-2r+1)}{2}+1.
\end{align*}
In the degenerate case, this strictly increases the lower bound of Theorem~\ref{th-conner} whenever $n-r>m>1$.
}

\section{Conclusion}
\label{sec:conclusion}
We have provided a Kalman-type controllability criterion and estimates on the length of shortest controllable switching sequences for discrete-time switched linear control systems. The estimates that we have obtained are tight in some relevant cases, but numerical examples show that they could be improved when the number of modes is small compared to the dimension of the system. Furthermore, the case in which the modes $A_i$ are degenerate is only partially treated here, and even the existence of an upper bound (depending on the dimension of the system) on the length of shortest controllable sequences remains unknown in this case. This, and further related questions, will be the subject of future research. 

\section*{Acknowledgements}
The authors thank Luca Greco for fruitful discussions and for drawing their attention to Example~\ref{example}.

\bibliographystyle{plain}
\bibliography{biblio.bib}

 \end{document}